\documentclass[a4]{article}
\usepackage{amsfonts}
\usepackage{amsmath}
\usepackage{amssymb}
\usepackage[mathscr]{eucal}
\usepackage{amsthm}
\theoremstyle{plain}
 \newtheorem{theorem}{Theorem}[section]
 
 \newtheorem*{theorem*}{Theorem}
 
 \newtheorem{lemma}[theorem]{Lemma}

\theoremstyle{definition}
 
\theoremstyle{remark}
 \newtheorem{remark}[theorem]{Remark}
 
\numberwithin{equation}{section}

\begin{document}
\title{On $Q$-polynomial association schemes of small class}

\author{Sho Suda\\
{\small Division of Mathematics, Graduate School of Information Sciences,} \\
{\small  Tohoku University,} \\
{\small 6-3-09 Aramaki-Aza-Aoba, Aoba-ku, Sendai 980-8579, Japan}}

\date{\today}

\maketitle

\begin{abstract}
We show an inequality involving the third largest or second smallest dual eigenvalues of $Q$-polynomial association schemes of class at least three.
Also we characterize dual-tight $Q$-polynomial association schemes of class three.
Our method is based on tridiagonal matrices and can be applied to distance-regular graphs as well. 
\end{abstract}
\section{Introduction}
$Q$-polynomial association schemes are defined by Delsarte in \cite{Delsarte} as a framework to study design theory uniformly, and
are studied in the last two decades 
from the viewpoints of structure theory \cite{CS,S1,S2,TT}, imprimitive cases \cite{DMM,LMO,MMW}, the dual version of Bannai-Ito conjecture \cite{MW}, hemisystems \cite{PW}, spherical designs \cite{Suda}. 

This concept is regarded as a dual object to distance-regular graphs (equivalently $P$-polynomial association schemes).
Many examples of $Q$-polynomial association schemes that are neither $P$-polynomial nor duals of translation $P$-polynomial association schemes are obtained from spherical designs \cite{BB,DGS}.
Small class $Q$-polynomial association schemes are attached to several combinatorial objects: linked systems of symmetric designs for $3$ class $Q$-antipodal case \cite{Dam,MMW,Suda1}, certain equiangular line sets for $3$ class $Q$-bipartite case \cite{MMW,Suda1}, real mutually unbiased bases for $4$ class, $Q$-antipodal and $Q$-bipartite case \cite{ABS,LMO,Suda1}.
Thus $Q$-polynomial association schemes of small class are of particular interest to research.  
The aim of this paper is to pursue this research direction further.

It was proven in \cite{KPY} that for a distance-regular graph of diameter $D\geq 2$ and distinct eigenvalues $k=\theta_0>\theta_1>\cdots>\theta_D$, the following inequality holds:
\begin{align}\label{eq:kpy}
(\theta_1+1)(\theta_D+1)\leq -b_1,
\end{align} 
here we use the standard notation of distance-regular graphs, see \cite{BCN}.
Moreover equality holds if and only if the diameter is two, meaning the graph is strongly regular. 

In \cite{JKT}, the following ``fundamental bound" for distance-regular graphs of diameter $D$ was given;
\begin{align}\label{eq:01}
\Big(\theta_1+\frac{k}{a_1+1}\Big)\Big(\theta_D+\frac{k}{a_1+1}\Big)\geq -\frac{k a_1 b_1}{(a_1+1)^2}.
\end{align}
A distance-regular graph is {\it tight} if it is nonbipartite and equality holds in \eqref{eq:01}.
Tight distance-regular graphs have been extensively studied in several papers, e.g. \cite{GT,JK,Pascasio1999JAC,Pascasio2003DM}.
In particular, 
Juri\v{s}i\'{c} and Koolen showed the following characterization in \cite[Theorem 3.2]{JK}: 
A nonbipartite distance-regular graph of diameter three is tight if and only if it is a Taylor graph.

Our main results are 
 Theorem~\ref{thm:41}, the dual result to \eqref{eq:kpy}, and
 Theorem~\ref{thm:51}, the dual result to \cite[Theorem 3.2]{JK}.

One of the methods to study $Q$-polynomial association schemes is investigation of the tridiagonal matrix of the first Krein matrix.
An advantage of the above is to give a unifying way to study distance-regular graphs as well as $Q$-polynomial association schemes.
In the present paper, we demonstrate how results on tridiagonal matrices derive a unifying proof of results both for distance-regular graphs and for $Q$-polynomial association schemes.
The original proofs of \eqref{eq:kpy} and \cite[Theorem 3.2]{JK} are based on combinatorial methods, but our method is based only on tridiagonal matrices obtained from the first intersection matrix or the first Krein matrix. 
It implies that our way presents alternative proofs of these results for distance-regular graphs.   
In fact our method yields new inequalities for eigenvalues of regular or distance-regular graphs in Theorems~\ref{thm:ineqgraph}, \ref{thm:31}.

\section{Preliminaries}
\subsection{Eigenvalues of tridiagonal matrices}
Let $D$ be a positive integer at least two.
Let $B=(b_{ij})_{0\leq i,j\leq D}$ be a nonnegative tridiagonal matrix with positive superdiagonal and subdiagonal entries of size $D+1$.
We set $\alpha_i=b_{ii}$ for $0\leq i\leq D$, $\beta_i=b_{i,i+1}$ for $0\leq i\leq D-1$ and $\gamma_i=b_{i,i-1}$ for $1\leq i\leq D$.
We also set $\gamma_0=0$ and $\beta_D=0$.
Throughout this paper, we consider the following condition:
\begin{align}\label{eq:cond}
\alpha_0=0,\quad \gamma_1=1, \alpha_i+\beta_i+\gamma_i=\kappa (0\leq i\leq D), 
\end{align}
where $\kappa$ is a positive number.
It is well known that all eigenvalues of $B$ are distinct and real, and $\kappa$ is the largest eigenvalue.
Let $\theta_0=\kappa>\theta_1>\cdots>\theta_D$ be the eigenvalues of $B$.
By \cite[p.123]{BCN}, $\theta_1,\ldots,\theta_D$ are the eigenvalues of the $D\times D$ tridiagonal  matrix
\begin{align}\label{eq:tri}
\tilde{B}=\begin{pmatrix}
-\gamma_1&\beta_1&&&\\
\gamma_1&\kappa-\beta_1-\gamma_2&\beta_2&&\\
&\gamma_2&\ddots &\ddots &\\
&&\ddots &\ddots &\beta_{D-1}\\
&&&\gamma_{D-1}&\kappa-\beta_{D-1}-\gamma_D
\end{pmatrix}.
\end{align}
We define $F_0(x)=1$ and $F_i(x)$ to be  
the characteristic polynomial of the principal submatrix of $\tilde{B}$ consisting of the first $i $ rows and first $i$ columns, for $1\leq i\leq D$. 
Then we can easily find that $F_1(x)=x+1$ and
\begin{align*}
F_i(x)=(x-\kappa+\beta_{i-1}+\gamma_i)F_{i-1}(x)-\beta_{i-1}\gamma_{i-1}F_{i-2}(x) 
\end{align*} 
for $i=2,\ldots, D$, and thus $F_D(x)=\prod_{i=1}^D (x-\theta_i)$.
By \cite[Remark~(5), p.203]{BI}, all roots of $F_i(x)$ are real and distinct for each $1\leq i\leq D$.
For $1\leq i\leq D$, let $\alpha_{i,1}>\cdots>\alpha_{i,i}$ be the roots of $F_i(x)$.
Since $F_1(x)=x+1$, $\alpha_{1,1}=-1$. 
The polynomial $F_{i-1}(x)$ has a root in the open interval $(\alpha_{i,j+1},\alpha_{i,j})$ for each $1\leq j\leq i-1$, namely $\alpha_{i,j+1}<\alpha_{i-1,j}<\alpha_{i,j}$ holds.
The following is used to prove Theorem~\ref{thm:1}.
\begin{lemma}\label{lem:ineq}
Let $a,b,c,d$ be real numbers satisfying $a\leq b< c\leq d$, 
and define $f(x)=(x-a)(x-d)$ and $g(x)=(x-b)(x-c)$.
Then $f(t)\leq g(t)$ holds for any $t\in[b,c]$.
Moreover equality holds for some $t\in(b,c)$ if and only if $a=b$ and $c=d$. 
\end{lemma}
\begin{proof}
Follows from the facts that $g(x)-f(x)$ is a polynomial of degree one and 
that $f(b)\leq g(b)$, $f(c)\leq g(c)$.
\end{proof}
The following theorem shows a relation between eigenvalues of $B$ and entries of $B$.
\begin{theorem}\label{thm:1}
Let $D$ be a positive integer at least two and $B$ a $(D+1)\times (D+1)$ tridiagonal matrix satisfying \eqref{eq:cond}.
Let $\theta_0=\kappa>\theta_1>\cdots>\theta_D$ be the eigenvalues of $B$.
\begin{enumerate}
\item $(\theta_1+1)(\theta_D+1)\leq -\beta_1$ holds with equality if and only if $D=2$.
\item Assume that $D\geq 3$ holds.
If $\beta_2+\gamma_3\geq \kappa+1$ holds, then $(\theta_1+1)(\theta_{D-1}+1)(\theta_D+1)\geq -\beta_1(\kappa+1-\beta_2-\gamma_3)$.
If $\beta_2+\gamma_3\leq \kappa+1$ holds, then $(\theta_1+1)(\theta_{2}+1)(\theta_D+1)\leq -\beta_1(\kappa+1-\beta_2-\gamma_3)$. 
Moreover equality holds in either case if and only if $D=3$. 
\end{enumerate}
\end{theorem}
\begin{proof}
(1):
Applying Lemma~\ref{lem:ineq} to $(a,b,c,d)=(\theta_D,\alpha_{2,2},\alpha_{2,1},\theta_1)$, $f(x)=(x-\theta_1)(x-\theta_D)$ and $g(x)=F_2(x)$, $f(t)\leq g(t)$ holds for any $t\in [\alpha_{2,2},\alpha_{2,1}]$.  
In particular, by $\alpha_{2,2}<\alpha_{1,1}=-1<\alpha_{2,1}$, $f(-1)\leq g(-1)$ i.e., $(\theta_1+1)(\theta_D+1)\leq-\beta_1$ holds. 

Moreover $(\theta_1+1)(\theta_D+1)=-\beta_1$ holds if and only if $\theta_1=\alpha_{2,1}$ and $\theta_D=\alpha_{2,2}$ hold by Lemma~\ref{lem:ineq}. 
This is equivalent to  
$F_2(x)=F_D(x)$ i.e., $D=2$.

(2):
Assume that $\beta_2+\gamma_3\geq \kappa+1$ holds.
This condition is equivalent to $F_3(-1)\leq0$, namely $\alpha_{3,2}\leq -1\leq \alpha_{3,1}$.
Using Lemma~\ref{lem:ineq} for $(a,b,c,d)=(\theta_{D-1},\alpha_{3,2},\alpha_{3,1},\theta_1)$,  $f(x)=(x-\theta_1)(x-\theta_{D-1})$ and $g(x)=(x-\alpha_{3,1})(x-\alpha_{3,2})$, $f(t)\leq g(t)$ holds for any $t\in [\alpha_{3,2},\alpha_{3,1}]$.
In particular, $(\theta_1+1)(\theta_{D-1}+1)\leq (\alpha_{3,1}+1)(\alpha_{3,2}+1)$ holds.
From $\theta_D\leq \alpha_{3,3}<-1$ we have $(\theta_1+1)(\theta_{D-1}+1)(\theta_D+1)\geq (\alpha_{3,1}+1)(\alpha_{3,2}+1)(\alpha_{3,3}+1)=-F_3(-1)=-\beta_1(\kappa+1-\beta_2-\gamma_3)$.
The statement under the assumption $\beta_2+\gamma_3\leq \kappa+1$ can be similarly proven. 

Equality holds in either case if and only if $\theta_1=\alpha_{3,1}, \theta_D=\alpha_{3,3}$, namely $D=3$.
\end{proof}

\subsection{Graphs}
Let $\Gamma$ be a connected simple $k$-regular graph with vertex set $V(\Gamma)$ and edge set $E(\Gamma)$.
We denote the adjacency matrix of $\Gamma$ by $A$ and
let $\theta_0=k>\theta_1>\cdots>\theta_D$ be the distinct eigenvalues of $A$ in descending ordering. 
Let $\partial$ be the path-length distance on $\Gamma$.
Assume $\Gamma$ is neither complete nor empty. 
Fix a vertex $x\in V(\Gamma)$, we define  
\begin{align*}
\Gamma_i(x)=\{y\in V(\Gamma) \mid \partial(x,y)=i\} 
\end{align*}
for $0\leq i\leq D_x$, where $D_x=\max\{\partial(x,y) \mid  y\in V(\Gamma)\}$.
The diameter of $\Gamma$ is defined to be $\max\{D_x\mid x\in V(\Gamma)\}$.
Then the graph $\Gamma$ has a distance partition $\pi(x)$ with respect to $x$ i.e.,  $\pi(x)=\{\Gamma_0(x),\Gamma_1(x),\ldots,\Gamma_{D_x}(x)\}$.
Let the characteristic matrix $S=S_x$ be the $|V(\Gamma)|\times (D_x+1)$ matrix with  $i$-th column as the 
characteristic vector of $\Gamma_i(x)$ for $0\leq i\leq D_x$. 
We define the quotient matrix $B=B(x)$ of $A$ with respect to $\pi(x)$ as $S^T S B= S^T A S$.
Note that the matrix $B$ is a nonegative tridiagonal matrix with positive superdiagonal and subdiagonal entries.
The entries of $B$ are denoted by $\alpha_i(x), \beta_i(x), \gamma_i(x)$.

Since the graph $\Gamma$ is $k$-regular, 
the quotient matrix $B$ satisfies the condition \eqref{eq:cond}.
We will then use the matrix $\tilde{B}$ defined in \eqref{eq:tri} to obtain a result for graphs in the next section.

For a vertex $x\in V(\Gamma)$, a graph $\Gamma$ is called {\it distance-regular around  $x$} if the numbers  $\gamma_i(x,y):=|\Gamma_{i-1}(x)\cap\Gamma_1(y)|,\alpha_i(x,y):=|\Gamma_{i}(x)\cap\Gamma_1(y)|,\beta_i(x,y):=|\Gamma_{i+1}(x)\cap\Gamma_1(y)|$ depend only on $x$ and the distance $i=\partial(x,y)$, not on the particular choice of $y\in \Gamma_i(x)$, for $0\leq i \leq D_x$.
The graph $\Gamma$ is called {\it distance regularised} if $\Gamma$ is distance-regular around all vertices in $\Gamma$.
The distance regularised graph is {\it distance-regular} if the parameters $\gamma_i(x,y),\alpha_i(x,y),\beta_i(x,y)$ depend only on $i=\partial(x,y)$, not on $x$ nor $y$.
A distance-regular graph of diameter two is called {\it strongly regular}.
The graph $\Gamma$ is called {\it distance-biregular} if the graph $\Gamma$ is distance-regularised, bipartite and the vertices in the same color class have the same intersection array.

It was proven in \cite{GS} that a distance regularized graph $\Gamma$ is either distance-regular or distance-biregular.
If the valencies on each bipartition are equal for a distance biregular graph, then it is distance-regular, see \cite[Lemma 1]{D}.
We will use the following lemma by Haemers.
\begin{lemma}\label{lem:DRG}(See \cite[Corollary 2.3,Theorem 7.3]{H})
Let $\Gamma$ be a connected regular graph having distinct eigenvalues $\theta_0>\theta_1>\cdots>\theta_D$ and let $B$ be the quotient matrix of the distance partition with respect to a vertex $x\in V(\Gamma)$ having distinct eigenvalues $\tau_0>\tau_1>\cdots>\tau_{D_x}$.
\begin{enumerate}
\item The eigenvalues of $B$ interlace the eigenvalues of $A$. In particular $\theta_1\geq \tau_1$ and $\tau_{D_x}\geq \theta_D$.
\item If $\theta_1=\tau_1$ and $\theta_D=\tau_{D_x}$ hold, then $\Gamma$ is distance-regular around $x$.   
\end{enumerate}
\end{lemma}

\section{Inequalities for eigenvalues of $k$-regular graphs}
Let $\Gamma$ be a regular, connected simple graph with valency $k$ with the adjacency matrix $A$ and the quotient matrix $B=B(x)$ of the distance partition $\pi=\{\Gamma_0(x),\Gamma_1(x),\ldots,\Gamma_{D_x}(x)\}$ for any $x\in V(\Gamma)$.

Let $\theta_0=k>\theta_1>\cdots>\theta_D$ be the distinct eigenvalues of $A$, and 
let $\tau_0>\tau_1 > \cdots > \tau_{D_x}$ be the eigenvalues of $B(x)$.
Since the graph $\Gamma$ is assumed to be $k$-regular, 
the quotient matrix $B(x)$ has the largest eigenvalue $k$.

By Theorem~\ref{thm:1}, we have $(\tau_1+1)(\tau_{D_x}+1)\leq -\beta_1(x)$.
Applying Lemma~\ref{lem:ineq} for $(a,b,c,d)=(\theta_D,\tau_{D_x},\tau_1,\theta_1)$ again, we have 
\begin{align}\label{eq:32}
(\theta_1+1)(\theta_D+1)\leq -\beta_1(x).
\end{align}

If equality is attained in \eqref{eq:32} for each $x\in V(\Gamma)$,
then $\tau_1=\theta_1$, $\tau_{D_x}=\theta_D$ and $D_x=2$ for each $x\in V(\Gamma)$. 
In particular the diameter of $\Gamma$ is two.
By Lemma~\ref{lem:DRG}, $\Gamma$ is distance-regular around all vertices in $V(\Gamma)$ with the same valency.
Therefore the graph $\Gamma$ is strongly regular.

Conversely when $\Gamma$ is strongly regular, it is easy to see that equality holds in \eqref{eq:32}. 
Therefore we have the following theorem.
\begin{theorem}\label{thm:ineqgraph}   
Let $\Gamma$ be a connected regular graph 
and let $\theta_0=k>\theta_1>\cdots>\theta_D$ be the distinct eigenvalues of $\Gamma$.
Then $(\theta_1+1)(\theta_D+1)\leq -\beta_1(x)$ holds for any vertex $x\in V(\Gamma)$.
Equality holds for all vertices if and only if $\Gamma$ is strongly regular.
\end{theorem}
The above is a generalization of Koolen, Park and Yu's inequality \eqref{eq:kpy} for regular graphs.

Applying Theorem~\ref{thm:1} to distance-regular graphs of diameter at least three, we have the following:
\begin{theorem}\label{thm:31}
Let $\Gamma$ be a distance-regular graph of diameter $D\geq3$ 
and let $\theta_0=k>\theta_1>\cdots>\theta_D$ be the distinct eigenvalues of $\Gamma$.
If $b_2+c_3\geq k+1$ holds, then $(\theta_1+1)(\theta_{D-1}+1)(\theta_D+1)\geq -b_1(k+1-b_2-c_3)$.
If $b_2+c_3\leq k+1$ holds, then $(\theta_1+1)(\theta_{2}+1)(\theta_D+1)\leq -b_1(k+1-b_2-c_3)$. 
Moreover equality holds in either case if and only if $D=3$. 
\end{theorem}
\section{Inequalities for dual eigenvalues of $Q$-polynomial association schemes}
The reader is referred to \cite{BI,MT} for the basic notations and information on $Q$-polynomial association schemes. 
Let $(X,\mathcal{R})$ be a $Q$-polynomial association scheme of class $D\geq2$.
Let $E_0,E_1,\ldots,E_D$ be the primitive idempotents of $(X,\mathcal{R})$.
Dual eigenvalues $\{\theta_h^*\}_{h=0}^D$ of $(X,\mathcal{R})$ are defined by 
$E_1=\tfrac{1}{|X|}\sum_{h=0}^D \theta_h^*A_h$.
We arrange the ordering of dual eigenvalues (i.e., the ordering of the adjacency matrices of the scheme) so that $\theta_0^*=m>\theta_1^*>\cdots>\theta_D^*$. 
We define the Krein parameters $q_{i,j}^h$ by $E_i\circ E_j=\frac{1}{|X|}\sum_{h=0}^D q_{i,j}^h E_h$, where $\circ$ denotes the entrywise product of matrices. 
By \cite[Theorem 4.1]{BI}, $B_1^*=(q_{1,j}^h)_{0\leq j,h\leq D}$ has the eigenvalues $\{\theta_h^*\}_{h=0}^D$.
We apply Theorem~\ref{thm:1} to the transpose of the tridiagonal matrix $B_1^*$ to obtain the following theorem.
\begin{theorem}\label{thm:41}
Let $(X,\mathcal{R})$ be a $Q$-polynomial association scheme of class $D\geq2$.
\begin{enumerate}
\item $(\theta_1^*+1)(\theta_D^*+1)\leq -b_1^*$ holds with equality if and only if $D=2$.
\item Assume that $D\geq 3$ holds.
If $b_2^*+c_3^*\geq m+1$ holds, then $(\theta_1^*+1)(\theta_{D-1}^*+1)(\theta_D^*+1)\geq -b_1^*(m+1-b_2^*-c_3^*)$.
If $b_2^*+c_3^*\leq m+1$ holds, then $(\theta_1^*+1)(\theta_{2}^*+1)(\theta_D^*+1)\leq -b_1^*(m+1-b_2^*-c_3^*)$. 
Moreover equality holds in either case if and only if $D=3$. 
\end{enumerate}
\end{theorem}
\begin{remark}
In \cite{C}, Cameron and Goethals constructed $Q$-antipodal $Q$-polynomial association schemes of class $3$, which are known as linked systems of symmetric designs, satisfying $(\theta_1^*+1)(\theta_3^*+1)=-b_1^*\frac{f}{f-1}$ with $f=2^{2m-1}$ for any positive integer $m$.

Therefore the above inequality (1) cannot be improved for the case of class $3$.
These association schemes are formally dual to the examples mentioned in \cite[p.2409, Remark]{KPY}.
\end{remark}
\section{Tight distance-regular graphs and dual-tight $Q$-polynomial association schemes}\label{sec:tight}
In \cite{JKT}, Juri\v{s}i\'{c}, Koolen and Terwilliger showed the following ``fundamental bound" for distance-regular graphs:
\begin{align}\label{eq:tight}
\Big(\theta_1+\frac{k}{a_1+1}\Big)\Big(\theta_D+\frac{k}{a_1+1}\Big)\geq -\frac{k a_1 b_1}{(a_1+1)^2}.
\end{align}
The same inequality above is proven by Pascasio \cite{Pascasio2003DM} in character algebras.
Applying Pascasio's result to $Q$-polynomial association schemes, we have the following inequality:
\begin{align}\label{eq:dualtight}
\Big(\theta_1^*+\frac{m}{a_1^*+1}\Big)\Big(\theta_D^*+\frac{m}{a_1^*+1}\Big)\geq -\frac{m a_1^* b_1^*}{(a_1^*+1)^2}.
\end{align}
A distance-regular graph is called {\it tight} if it is nonbipartite and equality holds in \eqref{eq:tight}, and a $Q$-polynomial association scheme is called {\it dual-tight} if it is not $Q$-bipartite and equality holds in \eqref{eq:dualtight}.  
Juri\v{s}i\'{c} and Koolen showed a characterization for the case of diameter $3$ in \cite{JK}.

The following theorem is a dual to the above characterization.
Our method is based only on tridiagonal matrices obtained from polynomial association schemes, so it gives an alternative proof of \cite[Theorem 3.2]{JK} and a unifying proof for distance-regular graphs and $Q$-polynomial association schemes.  
\begin{theorem}\label{thm:51}
Let $(X,\mathcal{R})$ be a $Q$-polynomial scheme of class $3$.
Then $(X,\mathcal{R})$ is dual-tight if and only if $(X,R_i)$ is the incidence graph of a symmetric design for some $i\neq 0$.
\end{theorem}
\begin{proof}
Suppose $(X,\mathcal{R})$ is dual-tight.
Then it follows that $a_3^*=0$; cf.\ \cite{T}\footnote{Tight distance-regular graphs of diameter $D$ satisfy $a_D=0$ \cite[Theorem~10.4]{JKT}. Dualizing the proof, we can show that dual-tight $Q$-polynomial association schemes satisfy $a_D^*=0$. See \cite{T} for the details.}.
Therefore the Krein matrix $B_1^*$ is 
$$B_1^*=\left(\begin{array}{cccc}
 0 & 1 & 0 & 0 \\
 m & m-b_1^*-1 & c_2^* & 0 \\
 0 & b_1^* & m-b_2^*-c_2^* & m \\
 0 & 0 & b_2^* & 0
 \end{array}\right). $$
 The characteristic polynomial $\phi(x)$ of $B_1^*$  is
$$\phi(x)=(x-m)(x^3+(-m+b_1^*+b_2^*+c_2^*+1)x^2+(b_1^*b_2^*+b_2^*+c_2^*-m b_2^*-m)x-m b_2^*) .$$
So we obtain 
\begin{align}
\theta_1^*+\theta_2^*+\theta_3^*&=m-b_1^*-b_2^*-c_2^*-1,\label{eq:51}\\
\theta_1^*\theta_2^*\theta_3^*&=mb_2^*.\label{eq:52}
\end{align}
Pascasio \cite{Pascasio1999JAC,Pascasio2003DM} showed that 
\begin{align}
\theta_1^*\theta_3^*=m\theta_2^*. \label{eq:53} 
\end{align}
Substituting (\ref{eq:53})  in (\ref{eq:52}), we have 
\begin{align}
{\theta_2^*}^2=b_2^* \label{eq:54}.
\end{align}
By the definition of dual-tightness, we have
\begin{align}
\theta_1^*\theta_3^*+\frac{m}{m-b_1^*}(\theta_1^*+\theta_3^*)=-\frac{m(b_1^*+1)}{m-b_1^*}. \label{eq:55}
\end{align} 
Substituting (\ref{eq:51}) and (\ref{eq:53}) in (\ref{eq:55}), we have by $a_1^*\neq0$
\begin{align}
\theta_2^*=-\frac{m-b_2^*-c_2^*}{m-b_1^*-1}. \label{eq:56}
\end{align}
Comparing the coefficient of $E_3$ in $E_1\circ E_1 \circ E_3$ in two ways and using $a_3^*=0$, 
we get $q_{2,3}^3=\frac{m(b_2^*-1)}{c_2^*}$.
Since this is a nonnegative real number, we have $b_2^*\geq 1$ and 
\begin{align}
{b_2^*}^2\geq b_2^*. \label{ineq:51} 
\end{align}
From (\ref{eq:54}), (\ref{eq:56}) and (\ref{ineq:51}), we have 
${b_2^*}^2\geq \frac{(m-b_2^*-c_2^*)^2}{(m-b_1^*-1)^2}$.
Hence $b_2^*(m-b_1^*-1)\geq m-b_2^*-c_2^*$, and we have
\begin{align}
b_2^*(m-b_1^*)\geq m-c_2^*, \label{ineq:52} 
\end{align}
with equality if and only if $b_2^*=1$.

Here, because $a_3^*=0$ and $q_{2,3}^3=\frac{m(b_2^*-1)}{c_2^*}$, we obtain
\begin{align*}
\frac{b_1^*b_2^*}{c_2^*}=m_3=q_{0,3}^3+q_{1,3}^3+q_{2,3}^3+q_{3,3}^3\geq 1+\frac{m(b_2^*-1)}{c_2^*},
\end{align*}
i.e.,
\begin{align}
m-c_2^*\geq b_2^*(m-b_1^*), \label{ineq:53} 
\end{align}
with equality if and only if $q_{3,3}^3=0$. 

By (\ref{ineq:52}) and (\ref{ineq:53}), $m-c_2^*=b_2^*(m-b_1^*)$ and $b_2^*=1$ hold.
Therefore $(X,\mathcal{R})$ is $Q$-antipodal i.e., $(X,\mathcal{R})$ is a linked systems of symmetric design \cite[Theorem 5.8]{Dam}.
And we obtain  $b_1^*=c_2^*$, so $(X,R_i)$ is the incidence graph of a symmetric design for some $i\neq 0$.

The converse follows from a straightforward calculation. 
\end{proof}
\section*{Acknowledgements}
A part of this work was done when the author stayed at POSTECH from August to September in $2009$ under the support of Japan International Science and Technology Exchange Center.
The author would like to thank Jack Koolen for drawing author's attention to this research and Hajime Tanaka for helpful discussions.

\end{document}